\newif\ifabstr
\abstrfalse

\ifabstr
\documentclass[10pt]{article}
\else
\documentclass[11pt]{article}
\fi

\newif\ifcomments %Show comments (suitable for checking abstract version length)
\commentstrue %Comments shown

\usepackage{amsmath,amssymb,amsthm}

\usepackage{url}
\usepackage{graphicx}
\usepackage[table]{xcolor} %A package for table with colored rows/columns 
\usepackage{hhline} % A supplementary package for tables

\ifabstr
\usepackage[small,compact]{titlesec}
\fi

\usepackage{graphicx}
\ifabstr\else
\usepackage{fullpage}
\fi
\ifabstr

\newtheorem{theorem}{Theorem}[section]
\newtheorem{lemma}[theorem]{Lemma}
\newtheorem{observation}[theorem]{Observation}
\newtheorem{corollary}[theorem]{Corollary}
\newtheorem{proposition}[theorem]{Proposition}
\newtheorem{question}[theorem]{Question}
\newtheorem{fact}[theorem]{Fact}
\newtheorem{claim}[theorem]{Claim}
\newtheorem{problem}[theorem]{Problem}
\newtheorem{definition}[theorem]{Definition}
\newtheorem{example}[theorem]{Example}
\newtheorem{remark}[theorem]{Remark}

\else %ifabstr
\newtheorem{theorem}{Theorem}[section]
\newtheorem{lemma}[theorem]{Lemma}

\newtheorem{corollary}[theorem]{Corollary}
\newtheorem{proposition}[theorem]{Proposition}

\theoremstyle{definition}
\newtheorem{definition}[theorem]{Definition}
\theoremstyle{remark}
\newtheorem{example}[theorem]{Example}
\newtheorem{remark}[theorem]{Remark}
\fi %ifabstr

\newcommand\R{\ensuremath{\mathbb{R}}}

\newcommand\Z{\ensuremath{\mathbb{Z}}}

\newcommand\NN{\ensuremath{\mathcal{N}}}
\newcommand\F{\ensuremath{\mathcal{F}}}
\newcommand\V{\ensuremath{\mathcal{V}}}%Venn diagram
\newcommand\K{\ensuremath{\mathcal{K}}}
\newcommand\X{\ensuremath{\mathcal{X}}}
\newcommand\A{\ensuremath{\mathcal{A}}}
\newcommand\B{\ensuremath{\mathcal{B}}}
\renewcommand\L{\ensuremath{\mathcal{L}}}

\newcommand\sect{\ensuremath{\textrm{reg}}}

\newcommand\one{\ensuremath{\mathbf{1}}}
\newcommand{\pth}[1]{\left( #1 \right)}
\newcommand\pthb[1]{\biggl(\,#1\,\biggr)}

\newcommand\Prob{\mathbb P}
\DeclareMathOperator{\avoid}{av}

\DeclareMathOperator{\MNF}{MNF}
\DeclareMathOperator{\At}{\mathbf{At}}

\newcommand{\marrow}{\marginpar{\boldmath$\longleftarrow$}}

\ifcomments
\newcommand{\jirka}[1]{\ifhmode\newline\fi\marrow \textsf{*** (JIRKA: ) #1\newline}}
\newcommand{\martin}[1]{\ifhmode\newline\fi\marrow \textsf{*** (MARTIN: ) #1\newline}}
\newcommand{\xavier}[1]{\ifhmode\newline\fi\marrow \textsf{*** (XAVIER: ) #1\newline}}
\newcommand{\pavel}[1]{\ifhmode\newline\fi\marrow \textsf{*** (PAVEL \& ZUZKA: ) #1\newline}}
\else
\newcommand{\jirka}[1]{}
\newcommand{\martin}[1]{}
\newcommand{\xavier}[1]{}
\newcommand{\pavel}[1]{}
\fi

\newcommand\coeff{\alpha}
\newcommand\coeffv{\boldsymbol{\alpha}}

\newcommand\xx{\boldsymbol{x}}
\newcommand\yy{\boldsymbol{y}}

\title{Simplifying inclusion - exclusion formulas}
\ifabstr
\begin{document}
\renewcommand\dagger{{**}}
\author{
%Xavier Goaoc\thanks{Universit\'e de Lorraine, Villers-l\`es-Nancy, F-54600, France. CNRS, Villers-l\`es-Nancy, F-54600, France, Inria, Villers-l\`es-Nancy, F-54600, France. A visit of this author in Prague was partially supported from Grant GRADR Eurogiga GIG/11/E023. E-Mail: \texttt{goaoc@loria.fr}}
  Xavier Goaoc\thanks{Universit\'e Paris-Est Marne-la-Vall\'ee,
    France.  This research was done while this author was affiliated
    with Inria, Project-team Vegas. A visit of this author in Prague
    was partially supported from Grant GRADR Eurogiga GIG/11/E023.
    E-Mail: \texttt{goaoc@univ-mlv.fr}}
 \and 
Ji\v{r}\'{\i} Matou\v{s}ek\thanks{Department of Applied Mathematics, Charles University,
   Malostransk\'{e} n\'{a}m.  25, 118~00~~Praha~1, Czech Republic, and Institute of  Theoretical Computer Science, ETH Zurich, 8092 Zurich,
 Switzerland. Supported by the ERC Advanced Grant No.~267165. Partially supported by the Charles University Grant GAUK 421511 and 
by Grant GRADR Eurogiga GIG/11/E023. E-mail: \texttt{matousek@kam.mff.cuni.cz}}
\and Pavel Pat\'ak\thanks{Department of Algebra, Charles University, Sokolovsk\'a 83, 186~75~Praha~8, Czech Republic.
Partially supported by the Charles University Grant GAUK 421511 and SVV-2012-265317. E-Mail:
\texttt{patak@kam.mff.cuni.cz}}\and
Zuzana Safernov\'a\thanks{Department of Applied Mathematics, Charles University,
   Malostransk\'{e} n\'{a}m.  25, 118~00~~Praha~1, Czech Republic. Supported by the ERC Advanced Grant No.~267165.
Partially supported by the Charles University Grant GAUK 421511. E-Mail: \texttt{\{zuzka, tancer\}@kam.mff.cuni.cz}}
\and Martin Tancer$^\S$}
\date{}
\else %\ifabstr
\usepackage{authblk} %A package for multiple affiliations
 \author[1,a]{Xavier Goaoc}
 \author[2,3,b,c,e]{Ji\v{r}\'{\i} Matou\v{s}ek}
 \author[4,c,d]{Pavel Pat\'ak}
 \author[2,b,c]{Zuzana Safernov\'a}
 \author[2,b,c]{Martin Tancer}
 
 \affil[1]{Universit\'e Paris-Est Marne-la-Vall\'ee}
 \affil[2]{Department of Applied Mathematics, Charles University,
   Malostransk\'{e} n\'{a}m.  25, 118~00~~Praha~1, Czech Republic.}
 \affil[3]{Institute of Theoretical Computer Science, ETH Zurich, 8092
   Zurich, Switzerland}
 \affil[4]{Department of Algebra, Charles University, Sokolovsk\'a 83,
   186~75~Praha~8, Czech Republic}
 \affil[a]{This research was done while this author was affiliated
   with Inria, Project-team Vegas. A visit of Xavier Goaoc in Prague
   was partially supported from Grant GRADR Eurogiga GIG/11/E023.}
 \affil[b]{Supported by the ERC Advanced Grant No.~267165.}
 \affil[c]{Partially supported by the Charles University Grant GAUK
   421511. } \affil[d]{Partially supported by the Charles University
   Grant SVV-2012-265317.}  \affil[e]{Partially supported by Grant
   GRADR Eurogiga GIG/11/E023.}
 \affil[*]{goaoc@univ-mlv.fr,\{matousek,patak,zuzka,tancer\}@kam.mff.cuni.cz}
\begin{document}
\date{\today}
\fi %\ifabstr

%\ifabstr
%\keywords{inclusion-exclusion formula, Venn diagram, abstract tube}
%\fi
%%%%%%%%%%%%%%%%%%%%%%%%%%%%%%%%%%%%%%%%%%%%%%%%%%%%%%%%%%%%%%%%%%%%%%%%%%

\maketitle
%{\bf Mathematics subject classification:} Primary 05A19, 68Q87;
%secondary 60C05, 68W20.

%\begin{abstract}
{
%\ifabstr
  \paragraph{Abstract.}%\else\fi
  Let $\F=\{F_1,F_2, \ldots,F_n\}$ be a family of $n$ sets on a ground
  set $S$, such as a family of balls in $\R^d$. For every finite
  measure $\mu$ on $S$, such that the sets of $\F$ are measurable, the
  classical \emph{inclusion-exclusion formula} asserts that
  $\mu(F_1\cup F_2\cup\cdots\cup F_n)=\sum_{I:\emptyset\ne
    I\subseteq[n]} (-1)^{|I|+1}\mu\Bigl(\bigcap_{i\in I} F_i\Bigr)$;
  that is, the measure of the union is expressed using measures of
  various intersections.  The number of terms in this formula is
  exponential in $n$, and a significant amount of research,
  originating in applied areas, has been devoted to constructing
  simpler formulas for particular families~$\F$.  We provide an upper
  bound valid for an arbitrary $\F$: we show that every system $\F$ of
  $n$ sets with $m$ nonempty fields in the Venn diagram admits an
  inclusion-exclusion formula with $m^{O(\log^2n)}$ terms and with
  $\pm1$ coefficients, and that such a formula can be computed in
  $m^{O(\log^2n)}$ expected time.  For every $\varepsilon>0$ 
  we also construct systems with Venn diagram of size $m$ for 
  which every valid inclusion-exclusion formula
  has the sum of absolute values of the coefficients at least $\Omega(m^{2-\varepsilon})$.}
%\end{abstract}

\section{Introduction}

One of the basic topics in introductory courses of discrete
mathematics is the \emph{inclusion-exclusion principle} (also called
the \emph{sieve formula}), which allows one to compute the number of
elements of a union $F_1\cup F_2\cup\cdots\cup F_n$ of $n$ sets from
the knowledge of the sizes of all intersections of the $F_i$'s.

We will consider a slightly more general setting, where we have a
ground set $S$ and a (finite) \emph{measure} $\mu$
on $S$; then the inclusion-exclusion principle asserts that for every
collection $F_1,F_2,\ldots,F_n$ of $\mu$-measurable sets, we have
\begin{equation}\label{eq:fullie}
  \mu\biggl(\bigcup_{i=1}^n F_i\biggr) = \sum_{I:\emptyset \neq I \subseteq
    [n]}(-1)^{|I|+1}\mu\biggl(\bigcap_{i \in I}F_i\biggr).
\end{equation}
(Here, as usual, $[n]=\{1,2,\ldots, n\}$ and $|I|$ denotes the
cardinality of the set $I$.) This principle not only plays a
fundamental role in various areas of mathematics such as probability
theory or combinatorics, but it also has important algorithmic
applications\ifabstr, \else. 
For instance, it provides simple methods for the
computation of volume or surface area of molecules in computational
biology~\cite{Mseed} and underlies, through efficient computation of
M\"obius transforms~\cite[Section 4.3.4]{aocp}, \fi \ifabstr e.g.,  \else\fi the best known
algorithms for several NP-hard problems including graph
$k$-coloring~\cite{set-partitioning}, travelling salesman problem on
bounded-degree graphs~\cite{travelling-salesman}, dominating
set~\cite{dominating-sets}, or partial dominating set and set
splitting~\cite{partial-dominating}.

\ifabstr \smallskip
\else
\medskip
\fi
The inclusion-exclusion principle involves a number of summands that
is exponential in $n$, the number of sets. In general this cannot be
avoided if one wants an \emph{exact} formula valid for \emph{every}
family $\F=\{F_1,F_2, \ldots,F_n\}$; see Example~\ref{ex:uniqueness} below for a
family for which Equation~\eqref{eq:fullie} is the only solution.
Yet, since this is a serious obstacle to efficient uses of
inclusion-exclusion, much effort has been devoted to finding
``smaller'' formulas. These efforts essentially organize along two
lines of research.

The first approach gives up on exactness and tries to
\emph{approximate} efficiently the measure of the union using the
measure of only \emph{some} of the intersections.
\ifabstr
See, e.g.,  
\else
The first results
of this flavor are the 
\fi
classical \emph{Bonferroni
  inequalities}~\cite{Bonferroni}\ifabstr.
We give a short overview of this line in the full version of this 
paper~\cite{ie_arxiv}.
\else.\footnote{These assert that if we
  omit all terms with $|I|>r$ on the right-hand side of
  (\ref{eq:fullie}), then we get an upper bound for the left-hand side
  for $r$ odd, and a lower bound for the left-hand side for $r$ even.
  The case $r=1$ is the often-used \emph{union bound} in probability
  theory.}
  It turns out that better approximations can be obtained by
replacing the coefficients $(-1)^{|I|+1}$ by other suitable numbers,
and such Bonferroni-type inequalities have been studied extensively;
see, e.g.,~\cite{Galambos}. Linial and Nisan~\cite{LinialNisan} and
Kahn et al.~\cite{KahnLinialSamorodnitsky} investigated how well
$\mu(F_1\cup\cdots\cup F_n)$ can be approximated if we know the
measure of all intersections $\bigcap _{i\in I} F_i$ for all
$I\subseteq [n]$ of size at most $r$.  Their main finding is that
having $r$ at least of order $\sqrt n$ is both necessary and
sufficient for a reasonable approximation in the worst case.  This
still leaves us with about $2^{\sqrt n}$ terms in approximate
inclusion-exclusion formulas.
\fi

The second line of research looks for ``small'' inclusion-exclusion
formulas valid for \emph{specific} families of sets. To illustrate the
type of simplifications afforded by fixing the sets, consider the
family $\F=\{F_1,F_2,F_3\}$ of
Figure~\ref{fig:example-simplification}.  Since $F_1\cap F_3=F_1\cap
F_2 \cap F_3$, Formula~\eqref{eq:fullie} can be simplified to
\ifabstr
\[\mu\pth{F_1 \cup F_2 \cup F_3}
=\mu(F_1)+\mu(F_2)+\mu(F_3)-\mu(F_1\cap F_2)-\mu(F_2 \cap F_3).\]
\vspace{-0.4cm}
\else
\[\mu\pth{F_1 \cup F_2 \cup F_3}
=\mu(F_1)+\mu(F_2)+\mu(F_3)-\mu(F_1\cap F_2)-\mu(F_2 \cap F_3).\]
\fi
\begin{figure}
\begin{center}
\includegraphics{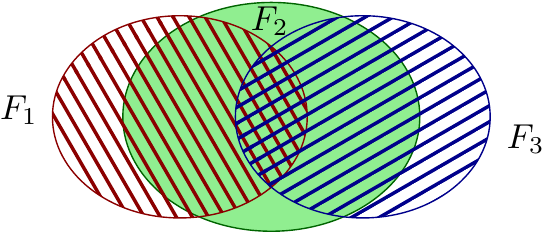}
\caption{Three subsets of $\R^2$ admitting a simpler
  inclusion-exclusion formula\label{fig:example-simplification}. The ground
set $F_1 \cup F_2 \cup F_3$ splits into six nonempty regions recognizable
by the filling pattern.}
\end{center}
\ifabstr
\vspace{-0.5cm}
\else
\fi
\end{figure}

More generally, let us consider a family $\F=\{F_1,F_2,\ldots,F_n\}$,
and let us say that a coefficient vector
\ifabstr
\[
\coeffv=(\coeff_I)_{\emptyset\ne I\subseteq [n]}\in\R^{2^n-1}
\]
\vspace{-0.6cm}
\else
\[
\coeffv=(\coeff_I)_{\emptyset\ne I\subseteq [n]}\in\R^{2^n-1}
\]
\fi
is an \emph{IE-vector for $\F$} if we have
\begin{equation}\label{e:valid-ie}
  \mu\biggl(\bigcup_{i=1}^n F_i\biggr) = \sum_{I\colon \emptyset \neq I \subseteq [n]}
  \coeff_I\mu\biggl(\bigcap_{i \in I}F_i\biggr)
\end{equation}
for every finite measure $\mu$ on the ground set of $\F$ (with all the
$F_i$'s measurable). Given $\F$, we would like to find an IE-vector
for $\F$, such that both the number of nonzero coefficients is small,
and the coefficients themselves are not too large. \ifabstr We learned this idea from from~\cite{AttaliEdelsbrunner} and refer to the monograph
of Dohmen~\cite{Dohmen} for an overview of this line of research. \else
This idea, which we
originally learned from~\cite{AttaliEdelsbrunner}, seems to originate
in the work of Kratky~\cite{Kratky} on families of disks in the plane,
and a systematic study of such simplifications was initiated by Naiman
and Wynn~\cite{NaimanWynn92,NaimanWynn97}. A simplified inclusion-exclusion
formula was also successfully used in an algorithm
of~Bj\"{o}rklund et al.~\cite{travelling-salesman}. We refer to the monograph
of Dohmen~\cite{Dohmen} for an overview of this line of research. 
\fi

Given a specific family $\F=\{F_1,F_2,\ldots,F_n\}$ of sets, how small
can we expect an inclusion-exclusion formula to be? 
\ifabstr \else This is, roughly
speaking, the question we tackle in this paper. \fi To formalize the
problem, we should specify how $\F$ is given. Let us
consider the \emph{Venn diagram} of $\F$\ifabstr. \else, which is the partition of
the ground set $S$ into equivalence classes according to the
membership in the sets of $\F$. \fi For each nonempty index set
$\tau\subseteq [n]$, we define the \emph{region} of $\tau$, denoted by
$\sect(\tau)$, as the set of all points that belong to the sets $F_i$
with $i\in \tau$ and no others (see
Figure~\ref{fig:example-simplification});
\[
\sect(\tau) = \pthb{\bigcap_{i\in \tau}F_i} \setminus \pthb{\bigcup_{i\not\in
\tau}F_i}.
\]
%\xavier{\cite{EdelsbrunnerRamos} use the term ``regions'' for our
%  ``sectors''. Is there any reason to use one over the other?}
%\jirka{I don't know who came up with sectors. I myself would
%slightly prefer regions.}

The \emph{Venn diagram} of $\F$ is then the collection of all subsets of
$[n]$ with non-empty regions; that is,
\[\V =\V(\F):= \{\tau \subseteq [n] \colon \sect(\tau) \neq \emptyset \}.\]
We regard the Venn diagram as a set system on
the ground set $[n]$; it is a ``dual'' of the set system $\F$.

\ifabstr As
%\else It is easy to see that, a\fi 
far as inclusion-exclusion formulas are
concerned, all points in a single region are equivalent; it only
matters which of the regions are nonempty.  Thus, in order to simplify
our formulations, we can assume that $\F$ is \emph{standardized},
meaning that the ground set equals the union of the $F_i$'s and each
nonempty region has exactly one point.  From an algorithmic point of
view, this amounts to a preprocessing step for $\F$, in which the part
of the ground set $S$ in each nonempty region is contracted to a
single point.
%The original formulation with standardized kept in the extended abstract only.
\else
%From an algorithmic point of view, we will need a good access to the Venn
%diagram.
We say that $\F$ is \emph{standardized} if the ground set equals the union of
the $F_i$'s and each nonempty region has exactly one point.
It is easy to see that, as far as inclusion-exclusion formulas are
concerned, all points in a single region are equivalent; it only
matters which of the regions are nonempty. Therefore assuming that $\F$ is
standardized does not mean a loss of generality. We will use this assumption
in the algorithmic part of our main result---Theorem~\ref{thm:bound}.
For general $\F$ this requires a preprocessing step for $\F$, in which the part
of the ground set $S$ in each nonempty region is contracted to a
single point.
\fi

Let $\F=\{F_1,F_2,\ldots,F_n\}$ be a family of sets and let $m$ denote
the size of $\V$ (which equals the size of the ground set for $\F$
standardized). 
\ifabstr
A linear-algebraic argument shows that \emph{every} (finite) family
$\F$ has an inclusion-exclusion formula with at most $m$ terms (see
Corollary~\ref{cor:linear}).
\else
A linear-algebraic argument shows that \emph{every} (finite) family
$\F$ has an inclusion-exclusion formula with at most $m$ terms (see
Corollary~\ref{cor:linear}) and $m$ terms are sometimes necessary (see
the beginning of Section~\ref{sec:lower}). The question of how small a
formula $\F$ admits may thus seem settled.
\fi
\ifabstr
However, the coefficients (see Example~\ref{ex:exponential}) might be exponentially large.
This requires computing the measure of the intersections with enormous precision
to obtain meaningful results. 
\else
There is, however, a
caveat: this linear-algebraic argument may yield \emph{exponentially
  large} coefficients (see Example~\ref{ex:exponential}). If we wanted
to use such a formula, we would need to compute with very high
precision, and perhaps more seriously, we would have to know the
measures of the various intersections with an enormous precision, in
order to obtain a meaningful result. This may be totally impractical,
e.g., in geometric settings where some physical measurements are
involved, or where the measures of the intersections are computed with
limited precision.\fi  Thus, we prefer inclusion-exclusion formulas where
not only the number of terms is small, but the coefficients are also
small.

Our main result is the following general upper bound; to our
knowledge, it is the first upper bound applicable for an arbitrary
family.

\begin{theorem}\label{thm:bound}
  Let $n$ and $m$ be integers and let $D=\lceil 2e \ln m\rceil \lceil
  2+\ln \frac{n}{\ln m} \rceil$. Then for every family $\F$ of $n$
  sets with Venn diagram of size $m$, there is an IE-vector $\coeffv$
  for $\F$ that has at most $\sum_{i=1}^D {\binom{n}{i}}\le m^{O(\ln^2
    n)}$ nonzero coefficients, and in which all nonzero coefficients
  are $\pm1$'s. Such an $\coeffv$ can be computed in $ m^{O(\ln^2 n)}$
  expected time if $\F$ is standardized.
\end{theorem}

The bound in this theorem is quasi-polynomial, but not polynomial, in
$m$ and $n$. We do not know if a polynomial bound can be achieved
 with $\pm1$ coefficients.  We have at least the following
lower bound, \ifabstr presented \else proved \fi in Section~\ref{sec:lower}, showing that
inclusion-exclusion formulas of \emph{linear} size are impossible in
general.

\begin{theorem}\label{t:lwb} 
  For any $\varepsilon>0$, for arbitrarily large values of $m$, there
  exists a family of sets with Venn diagram of size $m$ for which any
  IE-vector has $\ell_1$-norm at least
  $\Omega\left(m^{2-\varepsilon}\right)$.
%   For every $\varepsilon>0$, there are families of sets for which
%   every IE-vector has $\ell_1$-norm at least
%   $\Omega\left(m^{2-\varepsilon}\right)$, for infinitely many values
%   of $m$, where $m$ denotes the size of the Venn diagram.
\end{theorem}

We recall that the $\ell_1$-norm of a real vector $\xx\in\R^d$ 
is $\|\xx\|_1=\sum_{i=1}^d|x_i|$. The $\ell_1$-norm gives a lower
bound on the tradeoff between the number of nonzero coefficients
and their orders of magnitude (we recall that a formula
with $O(m)$ nonzero coefficients is always attainable,
the problem being that the coefficients may be too large).

\ifabstr
We only sketch proofs of the theorems above. In particular, we skip proofs of
auxiliary lemmas. They are proved in the full version~\cite{ie_arxiv}.

\else

\paragraph{Remark on $\ell_1$-norm minimization. }
A useful heuristic for finding ``small'' IE-vectors might be
to look for an IE-vector of minimum $\ell_1$-norm. In the linear-algebraic
formulation, this means finding a solution of $A\xx=\one$ of
minimum $\ell_1$-norm. 

It is well known that finding a solution of minimum $\ell_1$-norm
of a linear system can be done in polynomial time, via linear
programming. Several specialized algorithms for this problem
have also been developed, with better performance than direct
application of general-purpose LP solvers (see, e.g.,
\cite{yang-al} for a recent overview). However, in our setting
the number of columns of the matrix $A$ may be exponential
in $m$ and $n$, and so 
even the input for an $\ell_1$-norm minimizing algorithm
would be too large. 

There are linear programs with exponentially many variables
(and polynomially many constraints) that can still be solved
in polynomial time. For example, one may attempt, at least
for theoretical purposes, to solve the dual linear program 
by the ellipsoid method, provided that a separation oracle
is available. 

In our setting, the task of the separation oracle
can be formulated as follows in the setting of the original
(standardized) set system $\F=\{F_1,\ldots,F_n\}$: 
\emph{Given weights  $w_1,\ldots,w_m\in\Z$ of the points
and  threshold $c$, find a subset $I\subseteq [n]$, if one
exists, such that the sum of weights of the points
in $\bigcap_{i\in I}F_i$ is at least~$c$.}
Unfortunately, as was shown by Hoffmann et al.~\cite{gwop-team},
this problem is NP-complete not only for arbitrary set
systems, but also, e.g., for the case where each $F_i$
is the complement of a hexagon in the plane.
Thus, this approach doesn't
seem to lead to a polynomial-time algorithm for
finding an IE-vector of minimum $\ell_1$-norm
even for rather simple geometric settings.

\paragraph{Topological background.} In order to prove Theorem~\ref{thm:bound}
we need several basic notions from topological combinatorics. We
aim at a self-contained exposition that should make the proof accessible even
to a reader who is not familiar with topological methods (we use the
topological background mostly indirectly). For further reading we refer the
reader to sources such as~\cite{hatcher01, matousek03, munkres84}.
\fi

\section{Preliminaries}\label{sec:prel}

We consider a family $\F=\{F_1, F_2, \ldots, F_n\}$ of sets on a
ground set $S$, and assume that the $F_i$ are all distinct. Besides
the Venn diagram $\V$, we associate yet another set system with $\F$,
namely, the \emph{nerve}\footnote{This is the first notion from
  topological combinatorics that we need. Usually, a nerve also comes
  with an associated topological space that captures some of the
  properties of the underlying family $\F$.  In our case, a purely
  combinatorial description of the nerve is sufficient. We also
  emphasize that the condition $\sigma \neq \emptyset$ in the
  definition of $\NN(\F)$ is not a standard one but it is convenient
  for our purposes.
  % Usually, the nerve is also considered as a topological space and
  % its topological properties play an important role in properties of
  % $\F$. In our case, a purely combinatorial description of the nerve
  % is sufficient. We also emphasize that the condition $\sigma \neq
  % \emptyset$ in the definition of $\NN(\F)$ is not a standard one
  % but it is convenient for our purposes.
}  $\NN$ of $\F$:
\[ \NN=\NN(\F) := \biggl\{\sigma \subseteq [n] : \sigma \neq \emptyset,
  \bigcap_{i\in \sigma}F_i \neq \emptyset\biggr\}.\]
So both of $\NN$ and $\V$ have ground set $[n]$, and we have
$\V \subseteq \NN$.

Let us enumerate the elements of $\V$ as $\V=\{\tau_1,\tau_2, \ldots,
\tau_m\}$ in such a way that $|\tau_i| \le |\tau_j|$ for $i<j$, and
let us enumerate $\NN=\{\sigma_1, \sigma_2, \ldots, \sigma_{|\NN|}\}$
so that the sets of $\V$ come first, i.e., $\sigma_i=\tau_i$ for
$i=1,2,\ldots,m$.

In the introduction, we were indexing IE-vectors for $\F$
by all possible subsets $I\subseteq[n]$. But if $I$ is not
in the nerve, the corresponding intersection is empty, and
thus w.l.o.g.~we may assume that its coefficient is zero.
Thus, from now on, we will index IE-vectors $\xx$
as $(x_1, \dots, x_{|\NN|})$, where $x_j$ is the coefficient
of $\mu(\bigcap_{i\in\sigma_j}F_i)$.

\paragraph{IE-vectors from linear algebra.}
Let $A=(a_{jk})$ denote the $0$\,-$1$ matrix with $m$ rows and $|\NN|$
columns such that $a_{jk}=1$ if $\tau_j \supseteq \sigma_k$ and $a_{jk}=0$
otherwise. Let $\one$ denote the $m$-dimensional vector with all
entries equal to $1$.

\begin{lemma}\label{lem:system}
  $\xx \in \R^{|\NN|}$ is an IE-vector for $\F$ if and only if $A
  \xx=\one$.
\end{lemma}
\ifabstr

%\begin{proof}[Sketch of the proof]
%  Since each of the sets $\bigcup_{i=1}^n F_i$ and $\bigcap_{i
%    \in \sigma} F_i$, for $\sigma \subseteq [n]$, is a disjoint union
% of regions, we can rewrite Equation~\eqref{e:valid-ie}, with $\xx$
%  instead of $\coeffv$, to
%
 % \[
%  \sum_{j=1}^m \mu(\sect(\tau_j)) =
%    \sum_{j=1}^{m} \biggl(\sum_{k=1}^{|\NN|} a_{jk}x_k\biggr)
% \mu\pth{\sect(\tau_j)}.
%  \]
%
%  If $A\xx = \one$ then it is clear that equality holds for all
%  measures $\mu$. Conversely, if $\xx$ is an IE-vector for $\F$, then
 % a measure that is nonzero only on $\sect(\tau_j)$ shows that
%  $(A\xx)_j = 1$.  

%\end{proof}
\else
\begin{proof}
  A vector $\xx \in \R^{|\NN|}$ is an IE-vector for $\F$ if and only
  if for every finite measure $\mu$ on $S$ we have
  \begin{equation}\label{eq:iff}
  \mu\pth{\bigcup_{i=1}^n F_i} = \sum_{k=1}^{|\NN|}x_k\mu\pth{\bigcap_{i
      \in \sigma_k}F_i}.
  \end{equation}

  We first reformulate Equation~\eqref{eq:iff} using the regions of
  $\F$. The regions decompose $\bigcup_{i=1}^n F_i$ in a way
  that is compatible with the regions $\bigcap_{i \in \sigma}F_i$:
  \[ \bigcup_{i=1}^n F_i = \bigcup_{\tau \in \V} \sect(\tau) \quad \hbox{and
for all } \sigma
    \in \NN, \quad \bigcap_{i \in \sigma}F_i = \bigcup_{\tau
    \in \V \colon \tau \supseteq \sigma} \sect(\tau).\]
  Moreover, the regions are pairwise disjoint. Thus, for every finite measure
  $\mu$ on $S$ we have
  \[
  \mu\pth{\bigcup_{i=1}^n F_i} = \sum_{\tau \in \V} \mu\pth{\sect(\tau)}
  \quad \hbox{and for all }
   \sigma \in \NN, \quad \mu\pth{\bigcap_{i \in \sigma}F_i} = \sum_{\tau
    \in \V \colon \tau \supseteq \sigma} \mu\pth{\sect(\tau)},
  \]
  and Equation~\eqref{eq:iff} is equivalent to
  \[
  \sum_{\tau \in \V} \mu\pth{\sect(\tau)} = \sum_{k=1}^{|\NN|}x_k
  \pth{\sum_{\tau \in \V \colon \tau \supseteq \sigma_k}\mu\pth{\sect(\tau)}}.
  \]
  Using the orderings on $\V$ and $\NN$ and the definition of $A$ we
  obtain that $\xx \in \R^{|\NN|}$ is an IE-vector for $\F$ if and
  only if for every finite measure $\mu$ on $S$ we have
  \begin{equation}\label{eq:translation}
    \sum_{j=1}^m \mu\pth{\sect(\tau_j)} =
    \sum_{k=1}^{|\NN|}x_k\pth{\sum_{j=1}^m a_{j,k} \mu\pth{\sect(\tau_j)}} =
    \sum_{j=1}^{m} \pth{\sum_{k=1}^{|\NN|} a_{j,k}x_k} \mu\pth{\sect(\tau_j)}.
  \end{equation}

  Now, if $A\xx=\one$ then Equation~\eqref{eq:translation} trivially
  holds for all $\mu$ and $\xx$ is an IE-vector for $\F$. Conversely,
  assume that $\xx$ is an IE-vector for $\F$ and thus that
  Equation~\eqref{eq:translation} holds for all $\mu$.  For $1 \le j
  \le m$ we pick $p_j \in \sect(\tau_j)$ and define the measure
  $\mu_j:2^S \to \R$ by $\mu_j(T)=1$ if $p_j \in T$ and $0$ otherwise.
  Equation~\eqref{eq:translation} then specializes to
  \[ 1 = \mu_j\pth{\sect(\tau_j)} = \sum_{k=1}^{|\NN|}
  x_ka_{j,k}\mu_j\pth{\sect(\tau_j)} = \sum_{k=1}^{|\NN|} a_{j,k}x_k.\]
  This implies that $(A\xx)_j=1$. The statement follows.
\end{proof}

\begin{remark}
  In our definition a vector $\xx$ is an IE-vector for $\F$ if and
  only if Equation~\eqref{e:valid-ie} is valid for every finite
  measure.  As it follows from the proof of Lemma~\ref{lem:system}
  this definition is equivalent to extending this requirement to every
  (finitely additive) signed measure. (A signed measure satisfies the
  classical axioms of a measure with the exception that it may take
  negative values.)
\end{remark}
\fi

\begin{example}\label{ex:uniqueness}
  Let $S =
  2^{[n]}\setminus\{[n]\}$ and $F_i=2^{[n]\setminus \{i\}}$ for $i \in
  [n]$. It is easy to see that here $\NN=\V$
  and $A$ is a lower-triangular square matrix with $1$'s on the diagonal.
  Hence $A$ is invertible and, by Lemma~\ref{lem:system},
  $\F$ has a unique IE-vector, namely, the one from the standard
  inclusion-exclusion formula.
\end{example}

\begin{corollary}\label{cor:linear}
  For every finite family $\F$, there is a unique IE-vector
$\coeffv$ supported on $\V$ (that is, such that
$\coeff_I=0$ for $I\not\in\V$), and this $\coeffv$ 
has all entries integral.
\end{corollary}
\ifabstr\else
\begin{proof}
  Let $B$ be the $m \times m$ submatrix of $A$ consisting of the first
  $m$ columns of $A$. The IE-vectors for $\F$ supported on $\V$ are in
  one-to-one correspondence with the solutions of $B\yy=\one$.  Since
  $B$ is lower-triangular and has 1's on the main diagonal, it is
  nonsingular, and hence $B\yy=\one$ has exactly one solution.
  Moreover, since $B$ is a lower-triangular $0$-$1$ matrix, this
  solution is integral.
%  Observe that $B$ is a lower-triangular $0$-$1$ matrix with only
%  $1$'s entries on the diagonal. In particular, $B$ is non-singular
%  and the equation $B \xx = \one$ has a unique solution, which we
%  denote by $\hat{\xx}$, with integer coefficients. Let $\hat{\xx}^+$
%  stand for the vector in $\R^{|\NN|}$ whose first $m$ entries equal
%  that of $\hat{\xx}$ and whose remaining $|\NN| - m$ entries are $0$.
%  Since $A \hat{\xx}^+ = B\hat{\xx} = \one$, ${\xx}^+$ is an IE-vector
%  for $\F$ and the statement follows.
\end{proof}
\fi

\ifabstr The IE-vector of Corollary~\ref{cor:linear}, which can also
be obtained as the M\"{o}bius inverse of the constant function $1$ on
the poset $\V(\F)$, can, unfortunately, have exponentially large
coefficients as the following example shows.

\else
\begin{remark} 
  \label{r:recursive}  
  The matrix $B$ from the proof above can be regarded as the
  zeta-matrix of $\V$ ordered by inclusion. The vector $\coeffv$ from
  Corollary~\ref{cor:linear} can therefore be obtained via the
  M\"{o}bius inversion formula; see~\cite[Chapter 3]{Stanley}.
% Therefore $\coeffv$ from Corollary~\ref{cor:linear} can be obtained via M\"{o}bius  inversion formula; see~\cite[Chapter 3]{Stanley}.

  This description also yields a recursive formula for $\coeffv$ which we % utilize
  use in Section~\ref{sec:lower}. The condition $(B\yy)_j = 1$ translates
  as $\sum\alpha_{\tau} = 1$ where the sum is taken over all $\tau \in \V$ with
  $\tau \subseteq \tau_j$. That is, $\alpha_{\tau_j} = 1 - \sum \alpha_{\tau}$
  where the sum is taken over all $\tau \in \V$ properly contained in $\tau_j$.
\end{remark}

Unfortunately, the IE-vector with small support given by
Corollary~\ref{cor:linear} might have exponentially large
coefficients, as the following example shows.
\fi
%\ifabstr
%The vector $\coeffv$ from Corollary~\ref{cor:linear} can be obtained via
%M\"{o}bius inversion formula;
%see~\cite[Chapter 3]{Stanley}.
%\fi

\begin{example}\label{ex:exponential}
  Let $S = [5 \ell]$ for some positive integer $\ell$, and for $i \le
  \ell$, let $g(i)$ stand for the smallest integer $j\ge i$ divisible by $5$;
   that is $g(i) = 5\lceil i/5
  \rceil$. We consider the set system $\F = \{F_1, F_2, \dots,
  F_{5\ell}\}$ on $S$ given by
  $F_i = \{i\} \cup \{g(i) + 1, \dots, 5 \ell\}$. \ifabstr The
  largest coefficient of the IE-vector given by
  Corollary~\ref{cor:linear} is of order $4^{n/5}$.  \else
%  \[F_i = \{i\} \cup \{g(i) + 1, \dots, 5 \ell\}.\]
%
  Now $j \in F_i$ if and only if $i=j$ or $j > g(i)$. In particular,
  no two elements of $S$ belong to the same region and the number of
  regions of $\F$ is $m=|S|=5\ell$, which is also equal to the number
  $n$ of sets in $\F$: $n=m=5\ell$. The lower-triangular
   matrix $B$ from the proof of
  Corollary~\ref{cor:linear} has a simple structure in terms of $5
  \times 5$ blocks: the blocks on the diagonal are identity blocks, and
  the blocks below the diagonal are filled with $1$'s.
   Let $\hat\xx$ denote the
  solution of $B \xx = \one$. The first five rows yield $\hat{x}_1 =
  \hat{x}_2 = \dots = \hat{x}_5 = 1$. The next five rows imply that
  for $j=6,7,\ldots, 10$ we have
  \[ \hat{x}_1 + \hat{x}_2 + \dots + \hat{x}_5 + \hat{x}_j =1,\]
  and so $\hat{x}_6 = \hat{x}_7 = \dots = \hat{x}_{10} = -4$. A simple
  induction yields $\hat{x}_{i} = (-4)^{(g(i)/5) - 1}$.
  Altogether, the largest coefficient is of order $4^{n/5}$. (Replacing
  the constant $5$ by another constant $y$ yields a similar
  exponential growth with basis $(y-1)^{1/y}$; the choice $y=5$
  maximizes the basis of the exponent.)
\fi
%  \xavier{Changed $y^{1/(y+1)}$ to $(y-1)^{(1/y)}$ so that $y=5$
%    yields $4^{1/5}$ as before.}
% 
% Then $n = 5 \ell$ as well as $m = 5 \ell$ (since every element of $S$
% induces a nonempty region). More precisely, we have the following
% elements in $\V$:
% $$
% \tau_j = [g(j) - 5] \cup \{j\}
% $$
% since $j$ is contained in exactly those $F_i$ that $i \in \tau_j$. 
% 
% Let us now consider solution $\hat\xx$ of $B \xx = \one$. (Note that the order
% on $\tau_j$ we have chosen preserves the cardinality.) The first row of $B$
% contains $1$ on the first position a $0$ otherwise. Therefore $\hat{x}_1 = 1$.
% Similarly, $\hat{x}_2, \dots, \hat{x}_5 = 1$. The sixth row contains $1$ on
% first six positions and $0$ on the remaining positions. Therefore $x_1 + \cdots
% + x_6 = 1$ which yields $x_6 = -4$. Similarly, $x_1 + \cdots x_5 + x_s = 1$ for
% $s \in \{7,8,9,10\}$; and therefore $x_7 = x_8 = \cdots = x_{10} = -4$. Considering
% equation $x_1 + \cdots + x_{11}$ on the eleventh row we obtain $x_{11} = 16$.  
% Applying inductively this procedure, we obtain $x_{i} = (-4)^{(g(i)/5) - 1}$.
% Altogether, the largest coefficient is of order $4^{n/5}$. (The expression
% $y^{1/(y+1)}$ is maximized for $y = 4$ among integers; this explains `5' in our
% construction, since we want to achieve as large growth as possible.)
% 
% \martin{This part is the last part I was writing down (somewhat in rush).
% Therefore, some further improvements of this part are surely possible.}
% 
\end{example}

%\xavier{Gathered below all preliminaries on abstract tubes, both from
%  the introduction and section 3}

\paragraph{Abstract tubes.}
\ifabstr\else Naiman and Wynn~\cite{NaimanWynn92, NaimanWynn97}
started their study of simplified inclusion-exclusion formulas with
families $\F=\{F_1,F_2, \ldots, F_n\}$ that were tube-like in the
sense that $F_i \cap F_j \subseteq F_k$ for all $i \le k \le j$ (as in
our Figure~\ref{fig:example-simplification}). They then realized that
the simplifications found for these ``simple tubes'' hold in a broader
setting, leading them to introduce the more general notion of an
abstract tube.  This notion will also play an important role in our
considerations.

% While studying possible simplifications of
% inclusion-exclusion formulas, Naiman and Wynn~\cite{NaimanWynn92,
%   NaimanWynn97} started from families $\F=\{F_1,F_2, \ldots, F_n\}$
% that were tube-like in the sense that $F_i \cap F_j \subseteq F_k$ for
% all $i \le k \le j$ (as in our
% Figure~\ref{fig:example-simplification}). They have identified certain
% simplifications formulas for these ``simple tubes'' and they have also
% realized that these simplifications hold in a broader setting, leading
% them to introduce the more general notion of an abstract tube.  This
% notion will also play an important role in our considerations.
% 
%Naiman and Wynn~\cite{NaimanWynn92, NaimanWynn97} introduced the notion of an
%abstract tube that covers a certain class of simplifications. It will play an
%important role also in our considerations. 
\fi

\ifabstr\else
\begin{definition}
  \fi An \emph{(abstract) simplicial complex} with vertex set $[n]$ is
  a hereditary system of nonempty subsets of $[n]$.\footnote{As in the
    definition of the nerve, we exclude the empty set from the
    definition of a simplicial complex. This is again non-standard but
    convenient.}  An \emph{abstract tube} is a pair $(\F,\K)$, where
  $\F=\{F_1,F_2,\ldots, F_n\}$ is a family of sets and $\K$ is a
  simplicial complex with vertex set $[n]$, such that for every
  nonempty region $\tau$ of the Venn diagram of $\F$, the subcomplex
  \emph{induced} on $\K$ by $\tau$, $\K[\tau] := \{ \vartheta \in
  \K\colon \vartheta \subseteq \tau \},$ is contractible.\footnote{By
    \emph{contractible} we mean contractibility in the sense of
    topology; there is a topological space defined by $\K[\tau]$ and,
    roughly speaking, `contractible' means that this space can be
    continuously shrunk to a point. \ifabstr \else Readers not at ease
    with this notion may want to look at %the remark few lines    below.
    Remark~\ref{rk:Euler}.
\fi} \ifabstr\else
\end{definition}
\fi

As first noted by Naiman and Wynn~\cite{NaimanWynn92, NaimanWynn97},
if $(\F,\K)$ is an abstract tube, then
\begin{equation}\label{improvedie}
  \mu\pthb{\bigcup_{i=1}^n F_i} = \sum_{I \in
    \K}(-1)^{|I|+1}\mu\biggl(\bigcap_{i \in I}F_i\biggr).
\end{equation}
Moreover, truncating the sum yields upper and lower bounds in the
spirit of the Bonferroni inequalities (\cite{NaimanWynn97};
also see \cite[Theorem 3.1.9]{Dohmen}).

\ifabstr \else
\begin{remark}\label{rk:Euler}
An earlier, more permissive definition of abstract tubes
by~\cite{NaimanWynn92} had the weaker condition ``$\chi(\K[\tau])=1$''
instead of ``$\K[\tau]$ contractible,'' where 
$\chi$ is the \emph{Euler characteristic}.\footnote{The fact that all contractible complexes
  have the same Euler characteristic follows
from~\cite[Theorem~2.44]{hatcher01}. The fact that it equals $1$ can be
verified
on a point.} We recall
that for a simplicial complex $\L$ in our sense, the Euler characteristic is
defined as
$\chi(\L):=\sum_{\sigma\in\L}(-1)^{|\sigma|+1}$. 
In this setting, if $(\F,\K)$ satisfies $\chi(\K[\tau])=1$ for
every $\tau$, then
%Using this definition and Lemma~\ref{lem:system}, 
%the proof of 
(\ref{improvedie}) can be proven in a few lines, using Lemma~\ref{lem:system}.
 Indeed,  consider a simplicial complex $\K$ with vertex set $[n]$ and let
  $\xx \in \R^{|\NN|}$ stand for the vector with $x_k =
  (-1)^{|\sigma_k|+1}$ if $\sigma_k \in \K$ and $x_k = 0$ otherwise.
  Since
  \[ (A\xx)_j = \sum_{k \colon \sigma_k \subseteq \tau_j} x_k =
  \sum_{\sigma_k \colon \sigma_k \in \K[\tau_j]}(-1)^{|\sigma_k|+1},\]
we have $(A\xx)_j=\chi(\K[\tau_j])$. Thus, if all the $\K[\tau_j]$
have Euler characteristic 1, then $\xx$ is an IE-vector, and
(\ref{improvedie}) follows.
  
The stronger definition of abstract tubes involving
contractibility, as opposed to the Euler characteristic,
  was needed in order to  
  guarantee that
  truncations of Equation~\eqref{improvedie} also yield
  Bonferroni-type inequalities~\cite[Theorem~3.1.9]{Dohmen}.
\end{remark}
\fi

Small abstract tubes have been identified for families of
balls~\cite{NaimanWynn92, NaimanWynn97, AttaliEdelsbrunner} or
halfspaces~\cite{NaimanWynn97} in $\R^d$, and similar structures were
found for families of pseudodisks~\cite{EdelsbrunnerRamos}. We
establish Theorem~\ref{thm:bound} by proving that for every family of
sets there exists an abstract tube with ``small'' size that, in
addition, can be computed efficiently. We will use the following
sufficient condition guaranteeing that $(\F,\K)$ is an abstract tube;
it is a reformulation of~\cite[Theorem~4.2.5]{Dohmen}\ifabstr \else{}  
(for the reader's convenience we include a simple proof)\fi.  Let
$\MNF(\K)$ denote the system of all inclusion-minimal non-faces of
$\K$, i.e., of all nonempty sets $I\subseteq[n]$ with $I\not\in\K$ but
with $I'\in\K$ for every proper subset $I'\subset I$.

\begin{proposition}\label{prop:sufficient}
  Let $\F=\{F_1,F_2,\ldots, F_n\}$ be a family of sets with Venn
  diagram $\V$ and let $\K$ be a non-empty simplicial complex with vertex set
  $[n]$. If no set of $\V$ can be expressed as a union of sets in
  $\MNF(\K)$, then $(\F,\K)$ is an abstract tube.
\end{proposition}
\ifabstr
\else
\begin{proof}
  Let $\tau \in \V$ and let $a \in \tau$ such that $a$ belongs to no
  element of $\MNF(\K)$ contained in $\tau$. Our task is to show 
  that for every simplex $\vartheta \in \K[\tau]$ or $\vartheta = \emptyset$, 
  we have $\vartheta \cup \{a\} \in \K[\tau]$. A simplicial complex
  $\K[\tau]$ satisfying the mentioned condition is known as a \emph{cone} with
  \emph{apex} $a$. Since every cone is contractible, it remains to show the
  condition.
  
  If $\vartheta \cup \{a\} \notin \K[\tau]$, then
  $\vartheta \cup \{a\}$ contains some $\beta \in \MNF(\K)$; since
  $\vartheta \in \K[\tau]$, the face $\beta$ contains $a$, a
  contradiction.  
%  Thus $\vartheta \cup \{a\} \in \K[\tau]$ for every
%  $\vartheta \in \K[\tau]$. In other words, $\K[\tau]$ is a cone with
%  apex $a$. Since every cone is contractible, the statement follows.
\end{proof}
\fi

\ifabstr
\section{The upper bound: sketch of a proof of Theorem~\ref{thm:bound}}\label{sec:upper}
\else
\section{The upper bound: proof of Theorem~\ref{thm:bound}}\label{sec:upper}
\fi

\paragraph{Abstract tubes from selectors.}
Let $\F=\{F_1,F_2,\ldots, F_n\}$ be a family of sets, and let $\V$ be
the Venn diagram of $\F$. A \emph{selector} for $\V$ is a map $w\colon
\V \to [n]$ such that $w(\tau) \in \tau$ for every $\tau \in \V$. For
any selector $w$ for $\V$ we define the simplicial complex
\[ \K_w = \{\sigma \in \NN(\F) \colon \hbox{for all nonempty }
\vartheta \subseteq \sigma \hbox{ there is } \tau \in \V \hbox{ such
  that } w(\tau) \in \vartheta \subseteq \tau \}.\]
We observe that $(\F,\K_w)$ is an abstract tube since the complex
$\K_w$ satisfies the sufficient condition of
Proposition~\ref{prop:sufficient}.

%We observe that each selector for $\V$ provides an
%abstract tube for $\F$ (which satisfies the sufficient condition of
%Proposition~\ref{prop:sufficient}).

\begin{lemma}\label{lem:Kr}
  For any selector $w$ for $\V$, $(\F,\K_w)$ is an abstract tube.
% Let $\F=\{F_1,F_2,\ldots, F_n\}$, $\V=\V(\F)$, and let
% $w$ be a selector for $\V$. We define the simplicial complex
% $
% %\begin{equation}
% %\label{eq:Kr}
%   \K_w = \{\sigma \in \NN(\F) \colon 
% \hbox{for all nonempty } \vartheta \subseteq \sigma
%   \hbox{ there is } \tau \in \V \hbox{ such that } w(\tau) \in \vartheta \subseteq
%   \tau \}.
% %\end{equation}
% $
% Then $(\F, \K_w)$ is an abstract tube.
\end{lemma}
\ifabstr\bigskip \fi

\ifabstr
\else
\begin{proof} This is simple once the idea behind the definition
of $\K_w$ is explained. Namely, in the condition of  
Proposition~\ref{prop:sufficient} we want to prevent each set $\tau\in\V$
from being a union of minimal non-faces of the simplicial complex $\K$.
Our way of achieving that is to insist that every minimal non-face
$I$ contained in $\tau$ avoids the point $w(\tau)$; thus,
we consider the set system of ``admissible minimal non-faces''
\[
\B_w:=\{I\subseteq[n], I \neq \emptyset: \mbox{ if }I\subseteq\tau\in\V,
\mbox{ then } w(\tau) \notin I\}.
\]
Then the above definition of $\K_w$ can be interpreted as follows: a
simplex $\sigma\in\NN$ belongs to $\K_w$ if it contains no
$I\in\B_w$.\footnote{Note that for the formal verification, the
  condition $\sigma$ contains no $I \in \B_w$ can be written, in
  symbols, as follows: $\forall I \subseteq [n], I \neq
  \emptyset\colon ((\forall \tau \in \V\colon I \subseteq \tau
  \Rightarrow w(\tau) \notin I) \Rightarrow I \not \subseteq \sigma)$.
  This is equivalent to $\forall I \subseteq [n], I \neq
  \emptyset\colon I \subseteq \sigma \Rightarrow (\exists \tau \in
  \V\colon I \subseteq \tau \wedge w(\tau) \in I)$ which is just a
  transcription of $\sigma \in \K_w$.

} (Simplices outside $\NN$ can be ignored,
since their supersets cannot be contained in a set $\tau\in\V$.)
Therefore, all minimal non-faces of $\K_w$ belong
to $\B_w$ or lie outside $\NN$, and hence 
$(\F,\K_w)$ is an abstract tube by Proposition~\ref{prop:sufficient}.
%
%  As explained above, it follows from Proposition~\ref{prop:sufficient}
%  that if $\K$ is a simplicial complex over $[n]$ whose minimal non-faces are in
%  $\B_w$, then $(\F, \K)$ is  an abstract tube for $\F$. Let $\K \subseteq \NN$ denote
%  the simplicial complex over $[n]$ whose minimal non-faces (among non-faces
%belonging to $\NN$) are exactly  the inclusion-minimal elements of $\B_w$. 
%First, if $G \in 2^F
%\setminus N$ then $G$ contains no element $\tau \in H$ and $G$
%trivially belongs to $B_r$. Thus, $K \subseteq N$. 
%Moreover, $\sigma
%  \in \NN$ belongs to $\K$ if and only if no $\beta \in \B_w$ is a subset of
%$\sigma$, that is, if for every $\vartheta \subseteq \sigma$, $\vartheta \neq
%\emptyset$ we have $\vartheta
%\notin \B_w$ which is equivalent to
%%
%  \[ \hbox{there is } \tau \in \V \hbox{ such
%    that } \vartheta \subseteq \tau \hbox{ and } w(\tau) \in \vartheta.\]
%%
%  We recognize that $\K=\K_w$ and the statement follows.
\end{proof}
\fi
%\martin{The current formulation of the proof assumes that we have a notion $\K$
%is an abstract tube for $\F$. Otherwise, it should be reformulated.}

%\martin{I did some small changes in the proof above. In particular, I let $\K
%\subseteq \NN$ and considered only non-faces belonging to $\NN$. (Instead of
%letting $\K \subseteq 2^{[n]}$ and proving $\K \subseteq \NN$.) I think that it
%is a simplification. And it was somewhat necessary anyway since we already
%assumed that $\B_r \subseteq \NN$.}

\ifabstr
\vspace{-0.5cm}
\else
Let us remark that there is no loss of generality in passing
from the abstract tubes as in Proposition~\ref{prop:sufficient}
to those of the form $\K_w$. Indeed, if $\K$ satisfies the
condition of Proposition~\ref{prop:sufficient}, then every
$\tau\in\V$ contains at least one point that is not contained
in any minimal non-face $I$ of $\K$ with $I\subseteq\tau$, and such a point
can be chosen as $w(\tau)$---then we can easily check that
$\K_w\subseteq\K$. (It is sufficient to check that if $I$ is a minimal non-face
of $\K$, then it is also a non-face of $\K_w$. For this we point out that
such a minimal non-face $I$ of $\K$ belongs to the set $\B_w$ defined
above. Therefore it is a non-face of $\K_w$, possibly not a minimal one.) 
\fi

\paragraph{No large simplices in random $\K_w$.}
Let $\rho$ be a permutation of $[n]$. We define
a selector $w_\rho$ for $\V$ by taking $w(\tau)$
as the smallest element of $\tau$ in the linear ordering $\prec$
on $[n]$ given by $\rho(1) \prec \rho(2) \prec \cdots \prec \rho(n)$.
%
%\[ w_\rho(\tau) = \rho^{-1} (\min(\rho(\tau))) \quad \hbox{ for } \tau \subseteq 
%[n]. \]
%

For better readability we write $\K_\rho$ instead of $\K_{w_\rho}$. We 
want to show that for random $\rho$, $\K_\rho$ 
is unlikely to contain too large simplices, and  thus leads to
a small inclusion-exclusion formula.

\medskip

Let $\Gamma$ denote the incidence matrix of $\V$, that is,
the $0$-$1$ matrix with $m$ rows and $n$ columns where $\Gamma_{ij} =
1$ if and only if $j \in \tau_i$ (if the original system
$\F$ was standardized, then $\Gamma$ is the transposition
of the usual incidence matrix of $\F$). We also denote by $\Gamma_\rho$ the
matrix obtained by applying the permutation $\rho$ to the columns of
$\Gamma$: the $\rho(i)$th column of $\Gamma_\rho$ is the $i$th
column of $\Gamma$ and represents the incidences between permuted
$[n]$ and $\V$. \ifabstr The lemma below says \else We now argue \fi that if $\K_\rho$ contains a large
simplex, then $\Gamma_\rho$ contains a particular substructure. 

%\ifabstr
%\begin{lemma}
%  If $\rho(\tau) = \{i_1, i_2, \ldots, i_k\}$ for a
%    simplex $\tau$ in $\K_\rho$, the submatrix of $\Gamma_\rho$
%    formed by the columns $i_1,i_2, \ldots, i_k$ contains a row-permutation of
%    $J_k$, where $J_k$ is the $k\times k$ matrix with $1$'s on and above the
%    diagonal, and $0$ below it.
%\end{lemma}
%\else
We say that a row $R$ of $\Gamma_\rho$ is \emph{compatible} with a subset $I
\subseteq [n]$ if $R$ contains $1$'s in all columns with index in $I$
and $0$'s in all columns with index smaller than $\min(I)$. 

\begin{lemma}
\label{lem:compat}
  If $\rho(\tau) = \{i_1, i_2, \ldots, i_k\}$ for a
  simplex $\tau$ in $\K_\rho$, with $i_1<i_2<\ldots<i_k$, then for every $s \in
  \{1, 2, \ldots, k\}$ the matrix $\Gamma_\rho$ contains a row
  compatible with $\{i_{s}, i_{s+1}, \ldots, i_k\}$.
\end{lemma}
%\fi

\ifabstr
\else

\begin{proof}
  Let $s \in \{1, 2, \ldots, k\}$, let $I_s = \{i_{s}, i_{s+1},
  \ldots, i_k\}$, and let $\vartheta_s = \rho^{-1}(I_s)$. We refer to
  Figure~\ref{fig:Gamma_rho}. Since $\vartheta_s$ is a simplex of
  $\K_\rho$, there exists $\tau_{j_s} \in \V$ such that
  $w_\rho(\tau_{j_s}) \in \vartheta_s \subseteq \tau_{j_s}$ by
  %Lemma~\ref{lem:Kr}.
  definition of $\K_\rho$. Since $\vartheta_s \subseteq \tau_{j_s}$,
  we have $I_s = \rho(\vartheta_s) \subseteq \rho(\tau_{j_s})$, and
  hence the ${j_s}$th row of $\Gamma_\rho$ has $1$'s in all columns
  with index in $I_s$.  Since $w_\rho(\tau_{j_s}) \in \vartheta_s$,
  the set $\rho(\tau_{j_s})$ contains no $i$ with $i< i_s$ and the
  ${j_s}$th row of $\Gamma_\rho$ has $0$'s in all columns with index
  smaller than $i_s=\min(I_s)$. It follows that the ${j_s}$th row of
  $\Gamma_\rho$ is compatible with $I_s$.
\end{proof}

\begin{figure}
\begin{center}
\begin{tabular}{l|c|ccc|c|c|c|c|c|}
\multicolumn{1}{c}{} & 
\multicolumn{1}{c}{} 
& $i_1$ & $i_2$ & 
\multicolumn{1}{c}{$i_3$} &
\multicolumn{1}{c}{} & 
\multicolumn{1}{c}{ $i_4$} &
\multicolumn{1}{c}{} & 
\multicolumn{1}{c}{ $i_5$} &
\multicolumn{1}{c}{} 
 \\ %The first column ends here
\cline{2-10}
 \vdots & & & & & & & & & \\
\hhline{~---------} %This one is necessary, otherwise the grey cell overlaps
%the line
$j_3$ & \cellcolor[gray]{0.8} $0 \cdots 0$ & \cellcolor[gray]{0.8} 0 & \cellcolor[gray]{0.8}0 &
\cellcolor[gray]{0.8} 1 & $*\cdots*$ & \cellcolor[gray]{0.8} 1 & $*\cdots*$ & 
\cellcolor[gray]{0.8} 1 & $*\cdots*$ 
\\
\cline{2-10}
 \vdots & & & & & & & & & \\
\cline{2-10}
$j_1$ & $0 \cdots 0$ & 1 & 1 & 1 & $*\cdots*$ & 1 & $*\cdots*$ & 1 & $*\cdots*$ \\
$j_2$ & $0 \cdots 0$ & 0 & 1 & 1 & $*\cdots*$ & 1 & $*\cdots*$ & 1 & $*\cdots*$ \\
$j_4$ & $0 \cdots 0$ & 0 & 0 & 0 & $0\cdots 0$ & 1 & $*\cdots*$ & 1 &
$*\cdots*$ \\
$j_5$ & $0 \cdots 0$ & 0 & 0 & 0 & $0\cdots 0$ & 0 & $0 \cdots 0$ & 1 &
$*\cdots*$ \\
\cline{2-10}
 \vdots & & & & & & & & & \\
\cline{2-10}
\cline{2-10}
\end{tabular}
\end{center}

\caption{Illustration for Lemma~\ref{lem:compat}: If $\rho(\tau) = \{i_1, i_2,
  \ldots, i_5\}$ for a simplex $\tau$ of $\K_\rho$, then $\Gamma_\rho$
  must contain a row $j_s$ compatible with $\{i_s, i_{s+1}, \ldots,
  i_5\}$ for $s=1, 2, \ldots, 5$. The $j_3$ row is emphasized,
  constrained values appearing in grey; rows $j_s$ for other values of
  $s$ are represented consecutively for clarity, but they can appear in any
  order and non-consecutively.}
\label{fig:Gamma_rho}
\end{figure}

We will need the following inequality:

\begin{lemma}\label{c:ineq}
  Let $x_1, \dots, x_r$ be positive real numbers with $x_1 +
  \cdots + x_r\leq n$. Then
  \[
  \frac{x_1}{x_1 + \cdots + x_r} \cdot \frac{x_2}{x_2 + \cdots + x_r}
  \cdots \frac{x_{r-1}}{x_{r-1} + x_r} \leq \left( 1 - \root{r-1} \of
    {\frac{x_r}{n}}\right)^{r-1}.
  \]
\end{lemma}

\begin{proof}
  Let us set $y_\ell := x_\ell + x_{\ell+1} + \cdots + x_r$. Then we
  have
\begin{align*}
  \frac{x_1}{x_1 + \cdots + x_r} \cdot \frac{x_2}{x_2 + \cdots + x_r}
  \cdots \frac{x_{r-1}}{x_{r-1} + x_r} \quad & = \quad \frac{y_1 -
    y_2}{y_1} \cdot \frac{y_2 - y_3}{y_2} \cdots \frac{y_{r-1} -
    y_r}{y_{r-1}} \\
  &= \quad \left(1 - \frac{y_2}{y_1} \right) \cdot \left(1 -
    \frac{y_3}{y_2} \right) \cdots
  \left(1 - \frac{y_{r}}{y_{r-1}} \right) \\
  & \le \quad \left( \frac{1 - y_2/y_1 + 1 - y_3/y_2 + \cdots + 1 -
      y_r/y_{r-1}}{r-1}
  \right)^{r-1}\\
  & = \quad \left(1 - \frac{y_2/y_1 + y_3/y_2 + \cdots +
      y_r/y_{r-1}}{r-1}
  \right)^{r-1}\\
  & \le \quad \left(1- \root{r-1} \of{\frac{y_r}{y_1}} \right)^{r-1}\\
  & \le \quad \left(1- \root{r-1} \of{\frac{x_r}n} \right)^{r-1}.
\end{align*}
\end{proof}

\noindent

Now we aim at showing that for a random $\rho$, the condition
in Lemma~\ref{lem:compat} is unlikely to be satisfied for
large $k$. That condition prescribes the existence of $k$
rows in $\Gamma_\rho$ with a certain pattern.
In order to get a good bound for $k$, we won't actually look
for all of these  $k$ rows, but rather we will consider only
each $b$th of them, for a suitable integer parameter~$b$, and ignore
the rest.

Namely, we fix two parameters $r$ and $b$ with $1 < b < n$ and set $k=r b$ (we think
of $r \approx \ln n$ and $b \approx \ln m$). For
an $r$-element  index set $J\subseteq [m]$, 
let $\Gamma_\rho[J]$ denote the submatrix obtained from
$\Gamma_\rho$ by considering only the rows with indices in~$J$.
We say that a permutation $\rho$ is \emph{bad} for $J$
if there exists a $k$-element set of column indices $I = \{i_1, i_2, \ldots,
i_k\}$ with $i_1 < i_2 < \ldots < i_k$ such that for every $s \in \{1,
b+1, \ldots, (r-1)b+1\}$, the matrix $\Gamma_\rho[J]$ contains a row
compatible with $\{i_{s}, i_{s+1}, \ldots, i_k\}$.
Finally, we define $p_J$ as the probability that a random
permutation $\rho$ is bad for~$J$.

\begin{lemma}\label{l:estimate}
We have
  $ p_J \le (1 - (b/n)^{1/(r-1)})^{b(r-1)}$.
\end{lemma}

\begin{proof}
  Let $\rho$ be a bad permutation for $J$,
and let $I=\{i_{s}, i_{s+1}, \ldots, i_k\}$ be the corresponding
set of column indices.

  Let $\ell \in \{0, 1, \ldots, r-1\}$.
%  There are $r-\ell$ rows of $\Gamma_\rho[J]$ that are compatible with
%  a subset of $\{i_{\ell \cdot b+1}, i_{\ell \cdot b+2}, \ldots,
%  i_k\}$. 
  By the compatibility conditions we have that for $i < i_{\ell \cdot b+1}$, the $i$th
  column of $\Gamma_\rho[J]$ contains at most $\ell$ entries~$1$; see
  Figure~\ref{f:estimate}.
  Moreover, for $i \in \{i_{\ell\cdot b+1}, i_{\ell\cdot b+2}, \ldots,
  i_{(\ell+1)\cdot b}\}$, the $i$th column of $\Gamma_\rho[J]$
  contains exactly $\ell + 1$ entries $1$.
%  , since every row compatible with
%  $\{i_{s}, i_{s+1}, \ldots, i_k\}$ for $s\le \ell\cdot b+1$ has an
%  entry $1$ in these columns. 

\begin{figure}
%\begin{center}
\noindent\begin{tabular}{l|c|c|c|c|c|c|c|c|c|c|c|c|c|c|c|c|c|}
\multicolumn{1}{c}{} & 
\multicolumn{1}{c}{} & 
\multicolumn{1}{c}{$i_1$} &
\multicolumn{1}{c}{} & 
\multicolumn{1}{c}{$i_2$} &
\multicolumn{1}{c}{$\cdots$} &
\multicolumn{1}{c}{$i_b$} & 
\multicolumn{1}{c}{} & 
\multicolumn{1}{c}{ $i_{b+1}$} &
\multicolumn{1}{c}{$\cdots$} & 
\multicolumn{1}{c}{ $i_{2b}$} &
\multicolumn{1}{c}{} & 
\multicolumn{1}{c}{ $i_{2b+1}$} &
\multicolumn{1}{c}{$\cdots\cdots$} & 
\multicolumn{1}{c}{$i_{(r-1)b+1}$} &
\multicolumn{1}{c}{$\cdots$} &
\multicolumn{1}{c}{$i_{rb+1}$} &
\multicolumn{1}{c}{}
\\ %The first row ends here
\cline{2-18}
 & $0\cdots0$ &1& $*$ &1 &$\cdots$&1 & $*$ &1 &$\cdots$& 1 & $*$ & 1 &$\cdots\cdots$ 
& 1 &$\cdots$ & 1 & $*$\\
\cline{2-18}
 &
\multicolumn{1}{|l}{0}&
\multicolumn{5}{c}{$\cdots$}&
\multicolumn{1}{c|}{0} & 1 &$\cdots$ & 1 & $*$ & 1 &$\cdots\cdots$ & 1
		       &$\cdots$ & 1 & $*$ \\
\cline{2-18}
 &
\multicolumn{1}{|l}{0}&
\multicolumn{9}{c}{$\cdots$} &
\multicolumn{1}{c|}{0}& 1 & $\cdots\cdots$& 1 &$\cdots$& 1 & $*$ \\
\cline{2-18}
& \multicolumn{17}{|c|}{\vdots} \\
\cline{2-18}
 &
\multicolumn{1}{|l}{0}&
\multicolumn{11}{c}{$\cdots$}&
\multicolumn{1}{c|}{0} & 1 &$\cdots$& 1 & $*$ \\
\cline{2-18}
\end{tabular}
%\end{center}
\caption{Compatibility conditions in Lemma~\ref{l:estimate}. Only the rows of
  $J$ are shown and similarly as
before, and their order can be arbitrary.}
\label{f:estimate}
\end{figure}

  We now partition $[n]$ into $[n]=Q_0 \cup Q_1 \cup \ldots \cup Q_r$,
  where $Q_\ell$ consists of the indices of those columns of $\Gamma_\rho[J]$ that
  contain exactly $\ell$ entries $1$ (and $r-\ell$ entries $0$). In particular,
  from the discussion above, $|Q_\ell| \geq b$ for $\ell \in [r]$. For
  $\ell \in [r]$ and $p \in [b]$, let $g_\ell^{(p)}$ denote the
  $p$th smallest element of $\rho(Q_\ell)$. A necessary condition
  on $\rho$ is
  \[ g_1^{(b)} < g_2^{(1)} < g_2^{(b)} < g_3^{(1)} < \ldots <
  g_{r-1}^{(b)} < g_r^{(1)}.\]

  Now, let us assume that $\rho$ is a random permutation (uniformly chosen).
  For $\ell \in [r]$, let $E_\ell$ denote the event $E_\ell :=
  \{g_\ell^{(b)} < \min(g_{\ell+1}^{(1)}, g_{\ell+2}^{(1)}, \dots,
  g_{r}^{(1)})\}$, and we bound $p_J$ by the conditional probability
  \begin{equation} \label{e:conditional}
    p_J \le \Prob(E_1)\cdot \Prob(E_2| E_1) \cdot \Prob(E_3| E_1 \cap E_2) \cdots
    \Prob(E_{r-1}| E_1 \cap \cdots \cap E_{r-2}).
  \end{equation}
%
%  where the probability is taken over the choice of $\rho$ from the
 % uniform distribution on the permutations of $[n]$.  
For $\ell \in
  [r-1]$, $\Prob(E_\ell| E_1 \cap \cdots \cap E_{\ell - 1})$ is the
  probability that the $b$ smallest elements of $\rho(Q_\ell) \cup \rho(Q_{\ell +
    1}) \cup \cdots \cup \rho(Q_{r})$ belong to $\rho(Q_\ell)$. This probability is
  equal to
  \[ 
  \binom{|Q_\ell|}b \bigg/ \binom{|Q_\ell| + |Q_{\ell + 1}| + \cdots
    |Q_r|}b \le \pth{\frac{|Q_\ell|}{|Q_\ell| + |Q_{\ell + 1}| +
      \cdots |Q_r|}}^b.
  \]
  So, letting $x_\ell = |Q_\ell|$, Inequality~\eqref{e:conditional}
  implies
  \[ p_J \le \pth{\frac{x_1}{x_1+x_2+\ldots + x_r}\cdot\frac{x_2}{x_2+x_3+\ldots + x_r} \cdot \ldots \cdot \frac{x_{r-1}}{x_{r-1} + x_r}}^b
\le \pth{ 1 - \root{r-1} \of {\frac{|Q_r|}{n}}}^{b(r-1)},\]
the last inequality being  Lemma~\ref{c:ineq}.
%
%  \[ p_J \le \pth{ 1 - \root{r-1} \of {\frac{|Q_r|}{n}}}^{b(r-1)}\]
%
Then the lemma follows  using $|Q_r| \ge b$.
\end{proof}
%\martin{TODO: Here I am currently confused. As I am addressing the referee's
%remarks, I can at the moment deduce $|Q_\ell| \ge b$ for $b \in \{1,\dots,
%r-1\}$, but I am missing $|Q_r| \ge b$. I have to think about this (even if
%this was not OK, then $|Q_{r-1}| \ge b$ should be in principle sufficient, but
%it would require recomputing some values).}
\fi
\ifabstr
\begin{proof}[Sketch of a proof of Theorem~\ref{thm:bound}]
 \else
\begin{proof}[Proof of Theorem~\ref{thm:bound}]
\fi
Let $n$ and $m\ge 2$ be integers.\footnote{Note that the case $m=1$ is somewhat
trivial since every maximal face of $\NN$ belongs to $\V$, and thus there is an
IE-vector with a single non-zero coefficient, namely $1$, in this case.} Let $\F=\{F_1,F_2, \ldots, F_n\}$
  be a family of $n$ sets whose Venn diagram $\V$ has size $m$. 
\ifabstr 
We argue that if $\rho$ is a permutation of $[n]$ chosen uniformly at
random, the probability that all compatibility conditions of
Lemma~\ref{lem:compat} are satisfied for some $\{i_1, \dots, i_k\}$ is
at most $\frac12$ for $k := \lceil 2e \ln m\rceil \lceil 2+\ln
\frac{n}{\ln m} \rceil$. In particular, there exists a permutation
$\rho^*$ such that $\K_{\rho^*}$ contains no simplex of size $k$ (or
larger).  Lemma~\ref{lem:Kr} concludes the proof of
Equation~\eqref{improvedie}.
% Let $k := \lceil 2e \ln m\rceil \lceil 1+\ln \frac{n}{\ln m} \rceil$.
% We let $\rho$ to be a uniformly chosen random permutation of $[n]$.
% In the full version we prove that the probability that condition from
% Lemma~\ref{lem:compat} is satisfied for some $\{i_1, \dots, i_k\}$ is
% at most $\frac12$. In particular we can find $\rho^*$ such that
% $\K_{\rho^*}$ contains no simplex of size $k$ (or larger).
% Lemma~\ref{lem:Kr} concludes the proof of Equation~\eqref{improvedie}.
\else
Let
  $p(k)$ denote the probability that $\K_\rho$ contains at least one
  simplex of size $k$, where  $\rho$ is chosen uniformly at random among all
  permutations of $[n]$. From 
  Lemmas~\ref{lem:compat}~and~\ref{l:estimate}, for every $r>2$ and
  $b\ge 2$ we have
  \[ p(rb) \le \binom{m}{r}\pth{1 - \sqrt[r-1]{b/n}}^{b(r-1)} \le m^r
  e^{b(r-1) \ln \pth{1 - \root{r-1} \of{b/n}}} \le m^r
  e^{-b(r-1)\sqrt[r-1]{b/n}}.\]
  Assuming that $b \ge 2e\ln m$, we get
  $p(rb) \le m^{r-2e(r-1)\sqrt[r-1]{{b}/n}}$,
%  \[ p(rb) \le m^{r-2e(r-1)\sqrt[r-1]{\frac{b}n}},\]
%
  and choosing $r \ge 1+\ln\frac{n}{b}$, we obtain
  \[ \sqrt[r-1]{{b}/n} = e^{-\frac1{r-1} \ln \frac{n}b} \ge e^{-1} \quad
  \hbox{and} \quad p(rb) \le m^{2-r} \le \tfrac12.\]
  
Thus, with $D = \lceil 2e \ln m\rceil \lceil 2+\ln \frac{n}{\ln m} \rceil$
as in the theorem, we have $p(D)\le\frac12$ (note that setting $r = \lceil 2+\ln
\frac{n}{\ln m} \rceil$ implies $r > 2$ as required since $m \leq 2^n$).
%\martin{This I wanted to add as an explanation, but it is not true. I suggest
%switching to $r = \lceil 2+\ln \frac{n}{\ln m} \rceil$.}
So there exists a permutation
  $\rho^*$ of $[n]$ such that $\K_{\rho^*}$ contains no simplex of
  size $D$ (or larger).  By
  Lemma~\ref{lem:Kr}, $(\F,\K_{\rho^*})$ is an abstract tube and
  $\K_{\rho^*}$ has at most $\sum_{i=1}^D \binom{n}{i}$
  simplices. The IE-vector obtained from the abstract
tube $(\F,\K_{\rho^*})$ as in Equation~\eqref{improvedie}
is as claimed in the theorem.\fi

In order to actually compute a suitable coefficient vector,
we choose a random permutation $\rho$ and compute
$\K_\rho$\ifabstr\else~by the following incremental algorithm\fi.
\ifabstr
\else
We use two auxiliary set systems $\A$ and $\B$,
initialized to $\A=\B=\{\emptyset\}$ (the idea is that
$\B$ contains all the simplices of $\K_\rho$
found so far, and $\A\subseteq\B$ contains those for
which we still need to test one-element extensions).
In each step, we take some $\sigma\in\A$, remove it
from $\A$, and for each $i\not\in\sigma$, we test
whether $\sigma\cup\{i\}$ belongs to $\K_\rho$
(for this,  we just check
if there is $\tau \in \V$ such that $w_{\rho}(\tau) \in
  \sigma \cup \{i\} \subseteq \tau$; note that we have a direct access to $\V$
  in $O(m)$ time since $\F$ is standardized). Those 
$\sigma\cup\{i\}$ that pass this test are added 
to both $\A$ and $\B$. The algorithm finishes
either when $\A=\emptyset$ (in this case we
set $\K_\rho=\B\setminus\{\emptyset\}$ and return
the corresponding IE-vector), or when we first discover
a simplex $\sigma\in\K_\rho$ of size larger than~$D$.
In the latter case, we discard the current permutation
$\rho$, choose a new one, and repeat the algorithm. 

%  To compute such an abstract tube $(\F,\K_\)$ given the incidence
%  matrix $\Gamma$ we proceed as follows. Let $k = \lceil 2e \ln
%  m\rceil \lceil 1+ \ln \frac{n}{\ln m} \rceil$. We compute a random
%  permutation $\rho$ on $[n]$ and initialize our candidate simplicial
%  complex $\K_{can}$ with a single simplex, the empty
%  set.\footnote{For a uniform description, we temporarily allow an
%    empty set in a simplicial complex, although we were working
%    without empty sets} We then pick a simplex $\sigma$ in $\K_{can}$
%  and try to extend it, testing for each candidate extension $\sigma
%  \cup \{i\}$ whether it should be added to $\K_{can}$, that is,
%  whether there exists $\tau \in \V$ such that $w_{\rho}(\tau) \in
%  \sigma \cup \{i\} \subseteq \tau$. We continue until either we add a
%  simplex of size $k$ to $\K_{can}$, in which case we start over with
%  $\K_{can}=\{\emptyset\}$ and a new random $\rho$, or we cannot
%  extend any simplex of $\K_{can}$, in which case we are done
%  computing $\K_\rho$ and can output $\K = \K_{can}$. 
\fi
The choice of a random permutation $\rho$ takes $O(n \ln n)$ time and
$n$ random bits.  Accepting or rejecting a new simplex by brute-force
testing takes $O(mn)$ time. The expected number of times we have to
start over with a new permutation $\rho$ is $O(1)$. Altogether, the
expected running time of this algorithm is $O\pth{\binom{n}{D}mn} =
m^{O(\ln^2 n)}$.
\end{proof}

\ifabstr
 \section{The lower bound: sketch of a proof of Theorem~\ref{t:lwb}}\label{sec:lower}
\else
\section{The lower bound: proof of Theorem~\ref{t:lwb}}\label{sec:lower}

For every $m$ between $n$ and $2^n$ there exists a system of $n$ sets
with Venn diagram of size $m$ whose only IE-vector has $m$ nonzero
entries. Indeed, let $\K=\{\vartheta_1, \vartheta_2, \ldots,
\vartheta_m\}$ be a simplicial complex over $[n]$ such that $[n] =
\bigcup \K$ and $|\K|=m$. We define $F_i =\{t \in [m]\colon i \in
\vartheta_t\}$ for $1 \le i \le n$ and put $\F=\{F_1,F_2, \ldots,
F_n\}$. It can easily be checked that $\V(\F)=\NN(\F) = K$ and so, as
observed in Example~\ref{ex:uniqueness}, the matrix $A$ is square,
lower-triangular, and has $1$'s on the diagonal; thus, there is a
unique IE-vector for $\F$ and it has $m$ nonzero entries. In this
section we improve on this lower-bound.

We recall that by Corollary~\ref{cor:linear}, every set system $\F$
has a unique IE-vector with support in the Venn diagram $\V(\F)$. We
first argue that for some set systems constructed from lattices, this
IE-vector is the one with minimal $\ell_1$-norm. We then provide an
explicit construction, based on projective spaces over finite fields,
where the $\ell_1$-norm is near-quadratic in $m$.
%provide a family of set systems for which this IE-vector is the one
%with minimal $\ell_1$-norm.

\paragraph{Set systems from lattices.} 

We need to work with (finite) lattices as order-theoretic notions. A finite
partially ordered set $L$ is a \emph{lattice} if for every subset $S$ of $L$
there is the least upper bound for $S$ called the \emph{join} of $S$ and the
greatest lower bound called the \emph{meet} of $S$. A finite lattice always
contains the least element $0$. An \emph{atom} is an element $a \in L$ such
that $0$ is the only element lesser than $a$. A lattice is \emph{atomistic} if
each element is a join of some subset of atoms.

Given a finite atomistic lattice $L$ we construct the following set system $\F
= \F(L)$. Up to a relabeling, we can assume that the set of atoms 
of $L$ is $\At=\{1,2,\ldots n\}$.
For every atom $a\in \At$ we define $F_a:=\{x\in L\colon x\geq a\}$, and
for every $x\in L$ we set $\At_x:=\{a\in \At\colon a\leq x\}$.
For $\F=\{F_a\colon a\in \At\}$ we have $\V(\F)=\{\At_x\colon x\in L\setminus \{0\}\}$.
In particular, $\V(\F)$ equipped with the inclusion relation is isomorphic to
$L\setminus\{0\}$. %as a join semilattice.
Also note that $x$ is the join of $\At_x$ since $L$ is atomistic.

%\jirka{... 3) I also don't remember much about lattices - but can there be a finite
%lattice with an element that is not a join of atoms?}
%\pavel{... 3) Yes, for example a linear order on $>2$ elements}

\begin{lemma}\label{lem:lattice}
  Let $L$ be a finite atomistic lattice and $\F = \F(L)$ be the set
  system described above. Then among all IE-vectors for $\F$, the one
  with support in $\V(\F)$ has minimal $\ell_1$-norm.
\end{lemma}

\begin{proof} 
  Let $A$ be the matrix with rows indexed by $\V$ and columns indexed
  by $\NN=\NN(\F)$, as %introduced
  defined before Lemma~\ref{lem:system}, and let $B$ be the $m\times
  m$ submatrix consisting of the first $m$ columns of~$A$.

  We want to show that every column of $A$ is equal to a column of $B$.
  By the definition of $A$, this means that for every $\sigma\in\NN$
  we need to find some $\nu\in\V$ such that $\{\tau\in\V:
  \sigma\subseteq\tau\}= \{\tau\in\V:\nu\subseteq\tau\}$. 
  We set $s$ to be the join of $\sigma$. (Note that $\sigma$ is a subset of 
  $[n] = \At$ and, therefore, of $L$.) We aim to show that $\At_s$ is the
  required $\nu$. This way, we have obtained a $\nu \in \V$ such that the join
  of $\nu$ equals the join of $\sigma$ since $s$ is also the join of the atoms
  contained in $\At_s$.
  A set $\tau \in \V$ can be also described as $\At_x$ for some
  $x \in L \setminus \{0\}$ due to our description of $\V$. Then the condition
  $\sigma\subseteq\tau$ translates to $x \geq a$ for every $a \in \sigma$. This
  is equivalent with $x \geq s$ since $s$ is the join of $\sigma$. Similarly,
  $\nu\subseteq\tau$ translates to $x \geq a$ for every $a \in \nu$, which is
  again equivalent with $x \geq s$. Therefore, $\sigma\subseteq\tau$ if and only
  if $\nu\subseteq\tau$ as we need.
%  We observe that the join of all $\tau'\in\V$ with 
%  $\tau'\subseteq\sigma$ is such a $\nu$.
%  Indeed, from $\sigma = \bigcup\{\tau'\in\V:\tau'\subseteq \sigma\}$, we deduce
%  $\sigma \subseteq \nu$ which implies the `$\supseteq$' inclusion. We also
%  deduce the other inclusion again from $\sigma = \bigcup\{\tau'\in\V:\tau'\subseteq \sigma\}$ and from the definition of the join.

  Hence every column of $A$ occurs in $B$ as asserted.  It follows
  that every solution of $A\xx=\one$ can be transformed to a solution
  of $B\yy=\one$ with the same or smaller $\ell_1$-norm (if $k$ is the
  index of a column outside $B$ with $x_k\ne 0$, and that $k$th column
  equals the $j$th column of $B$, then we can zero out $x_k$ while
  replacing $x_j$ with $x_j+x_k$).  Since $B\yy=\one$ has a unique
  solution, it has to be a solution of minimum $\ell_1$-norm as
  claimed.
\end{proof}

\paragraph{Construction based on projective spaces.}

Let $q$ be a power of a prime number. Let $P$ be a projective space of
dimension $d$ over the finite field $F_q$.  That is, the points of $P$
are all $1$-dimensional subspaces of the vector space $F_q^{d+1}$, and
$k$-dimensional subspaces of $P$ correspond to $(k+1)$-dimensional
linear subspaces of $F_q^{d+1}$.  We let $L$ be the lattice of all
subspaces of $P$ (including the zero one, of projective dimension
$-1$, as zero), where the join of subspaces of $P$ corresponds to the
(projective) span and the meet corresponds to the intersection. It is
easy to check (and well known) that $L$ is an atomistic lattice.

 We
obtain our lower bound from the family $\F = \F(L)$ and so, according to
Lemma~\ref{lem:lattice}, we need only to compute the size of $\V(\F)$
and the $\ell_1$-norm for the IE-vector with support in $\V$ to
provide a lower bound.
% and thus we obtain the required family $\F = \F(L)$.
% 
% It is easy to check (and well known) that $L$ is an atomistic lattice
% in this case where the join of subspaces of $P$ corresponds to the
% (projective) span and the meet corresponds to the intersection.
% According to Lemma~\ref{lem:lattice}, we need only to compute the size
% of $\V(\F)$ and the $\ell_1$-norm for the IE-vector with support in
% $\V$ to provide a lower bound.
%
In order to do so, we need to work with $q$-binomial coefficients.
\begin{definition}[$q$-binomial coefficients]\leavevmode
\begin{enumerate}
\item[\rm(1)] Given a positive integer $k$, %and $q$ real
  we define $[k]_q := 1 + q + q^2 + \ldots + q^{k-1}$.
\item[\rm(2)] Given nonnegative integers $n$ and $k$ with $n \geq k$,
%  and $q \neq -1$ real,
  we define
    \[\binom{n}{k}_q := \frac{[n]_q[n-1]_q[n-2]_q\cdots
    [n-k+1]_q}{[1]_q[2]_q[3]_q\cdots [k]_q}.\]
 \end{enumerate}
\end{definition}

We remark that it is well known that $\binom{n}{k}_q$ is actually a
polynomial in $q$ since the division is
exact. % (in particular the definition can be extended for $q=-1$ as well).
From the definition above we deduce that the leading term of
$\binom{n}{k}_q$ is $q^{k(n-k)}$.  We also need the following facts
regarding $q$-binomial coefficients to finish the calculations. See,
for example,~\cite{Cohn04} and~\cite{PolyaAlexanderson71}.
\begin{lemma}\leavevmode
\label{l:q_binom}
 \begin{enumerate}
  \item[\rm(1)] The number of $k$-dimensional subspaces of a $d$-dimensional projective space over $F_q$ is $\binom{d+1}{k+1}_q$.
  \item[\rm(2)]{\rm (The Cauchy binomial theorem) }
        \[\sum_{i=0}^k q^{\frac{i(i-1)}{2}}\binom{k}{i}_q t^i =
	\prod_{i=0}^{k-1}(1 + tq^i).\]
 \end{enumerate}
\end{lemma}
%\martin{I have added two references. I did not really find a nice survey where
%both of the facts would be present/proved. Each of the cited references
%contains only one of them. For collection of facts without proofs wikipedia and
%wolfram seem indeed the best. I do not know whether we should cite them as
%well. (It is perhaps not necessary in my not very strong opinion.)}

Now we can finally estimate the size of $|\V(\F)|$ and the $\ell_1$-norm of the
resulting IE-formula.

%Because there are infinitely many possible values of $q$, we consider $q$-binomial coefficients as a polynomials in $q$
%and obtain the following asymptotic behavior.
\begin{lemma}\leavevmode
 \begin{enumerate}
  \item[\rm(1)] The number of nonempty subspaces of $P$, that is, the size
    of $\V(\F)$ is $\Theta\Bigl(q^{\left\lfloor{(d+1)^2}/{4}\right\rfloor}\Bigr)$.
  \item[\rm(2)] In the (unique) IE formula for $\F$, 
the coefficients of the subspaces of dimension $k$ are all equal
    to $(-1)^kq^{\frac{k(k+1)}{2}}$.
  \item[\rm(3)] The $\ell_1$-norm of the resulting IE-formula is $\Theta\Bigl(q^\frac{d(d+1)}{2}\Bigr)$.
 \end{enumerate}
\end{lemma}
\begin{proof}
  Concerning statement~(1), Lemma~\ref{l:q_binom}(1) implies that 
$$
|\V(\F)| = |L \setminus\{0\}| = \sum\limits_{k=0}^d \binom{d+1}{k+1}_q,
$$
which is a polynomial in $q$. Since we know that the leading term of
$\binom{d+1}{k+1}_q$ is $q^{(k+1)((d+1)-(k+1))}$, we deduce that the middle
$q$-binomial coefficient(s) has/have the leading term of the highest power. That is, 
the leading term of the polynomial above equals
$q^{\left\lfloor{(d+1)^2}/{4}\right\rfloor}$ or
$2q^{\left\lfloor{(d+1)^2}/{4}\right\rfloor}$ (depending on the parity of $d$) as we need.
%the middle $q$-binomial coefficient $\binom{d+1}{\lfloor (d+1)/2\rfloor}_q$ has clearly the highest power of $q$ among all $\binom{d+1}{k}_q$, namely 
%       \[q^{d + (d-1) + \ldots + \left((d + 1)-\lfloor (d+1)/2\rfloor\right) - \left(1 + 2 + \ldots + \left(\lfloor (d+1)/2\rfloor - 1\right)\right)} = q^{\left\lfloor{(d+1)^2}/{4}\right\rfloor}.\]

We prove statement (2) by induction. The statement clearly holds for
$k=0$. Suppose that it is valid for all $i<k$.  Using
Lemma~\ref{l:q_binom}(1) again, we see that every subspace of
dimension $k$ has $\binom{k+1}{i+1}_q$ subspaces of dimension $i$.
Therefore, using the recursive formula from Remark~\ref{r:recursive},
the coefficient of this subspace has to be
 \[1 - \sum_{i=0}^{k-1} (-1)^iq^{\frac{i(i+1)}{2}}\binom{k+1}{i+1}_q=
 \sum_{j=0}^{k} (-1)^j q^{\frac{j(j-1)}{2}}\binom{k+1}{j}_q.\]
 However, using the Cauchy binomial theorem for the second equality
 below, this sum equals
 \begin{align*}
   \sum_{j=0}^{k} (-1)^j
   q^{\frac{j(j-1)}{2}}\binom{k+1}{j}_q & =  
   (-1)^kq^{\frac{k(k+1)}{2}} + \sum_{j=0}^{k+1} (-1)^j
   q^{\frac{j(j-1)}{2}}\binom{k+1}{j}_q \\
   & = (-1)^kq^{\frac{k(k+1)}{2}} 
     + \prod_{j=0}^k(1 - q^j) \\
     & =  (-1)^kq^{\frac{k(k+1)}{2}} + 0,
 \end{align*} 
 which concludes the induction.

 It remains to prove statement~(3). Using statement~(2), we deduce
 that the $\ell_1$-norm of the resulting formula equals
$$
\sum\limits_{k=0}^d q^{\frac{k(k+1)}2} \binom{d+1}{k+1}_q.
$$
The leading term of this polynomial (in $q$) is
$2q^{\frac{d(d+1)}{2}}$. Indeed, the leading term of
$q^{\frac{k(k+1)}2} \binom{d+1}{k+1}_q$ equals $q^{\frac{k(k+1)}2 +
  (k+1)(d-k)}$ and is greatest for $k=d$ and $k=d-1$.
%, which can be
%easily deduced from the fact that the leading term of $q^{\frac{k(k+1)}2}
%\binom{d+1}{k+1}_q$ equals $q^{\frac{k(k+1)}2 + (k+1)(d-k)}$ (the
%exponent is the greatest for $k=d$ and $k=d-1$).
\end{proof}

%\martin{In the original setting the lower bound was obtained by evaluation at
%$k=d$. I think that in general this is not sufficient because eventually this
%could be also subtracted in a polynomial for another $k$. Thus I emphasized
%what is the leading term. (Then the change to $\Theta$ is for free.)}
%\pavel{I agree your approach is nicer.
%Well, it was sufficient, since we compute $\sum_i |p_i(q)|$ for $q$ huge. If for $i=k$ the polynomial is $p_i = q^{d(d-1)/2}$ (exactly),
%this power cannot disappear in $p_i$. If this power disappears due to some $-p_i$ inside some $p_j$, then this $p_j$ contains 
%a bigger power of $q$ and due to the absolute value, highest such power has to appear with positive sign.
%(All because we are only interested in asymptotic behaviour, and hence past some point lower powers of $q$ do not play
%any role). Hence $\Omega(q^{d(d-1)/2})$.
%}

\begin{proof}[Proof of Theorem \ref{t:lwb}]
  Fix $\varepsilon>0$ and let $d>2/\varepsilon$ be some integer, chosen to be odd for simplicity. 
% simplicity let $d>2/\varepsilon$ be odd. 
  Recall that the above analysis holds for any $q$ that is a prime
  power, so $q$ can be chosen arbitrarily large. The set system
  $\F(L)$ consists of $n = [d+1]_q = \Theta(q^d)$ sets. The Venn
  diagram $\V(\F(L))$ has size $m =
  \Theta\left(q^{\frac{(d+1)^2}{4}}\right)$ and the $\ell_1$ norm of
  the formula supported by the Venn diagram is
  \[\Theta\left(q^{\frac{d(d+1)}{2}}\right) =
  \Theta\left(m^{\frac{4}{(d+1)^2}\cdot \frac{d(d+1)}{2}}\right) =
  \Theta\left(m^{2-\frac{2}{d+1}}\right) \geq
  \Omega\left(m^{2-\varepsilon}\right).\]
  Lemma~\ref{lem:lattice} ensures that this formula 
  minimizes the $\ell_1$ norm.
% , and $q$ be a power of a prime number.
%  For the set system as above, we  have
%  $n = [d+1]_q = \Theta(q^d)$, $m = \Theta\left(q^{\frac{(d+1)^2}{4}}\right)$, 
%  and the minimal $\ell_1$ norm is $\Theta\left(q^{\frac{d(d+1)}{2}}\right) = 
%  \Theta\left(m^{\frac{4}{(d+1)^2}\cdot \frac{d(d+1)}{2}}\right) =
%  \Theta\left(m^{2-\frac{2}{d+1}}\right) \geq \Omega\left(m^{2-\varepsilon}\right)$.
\end{proof}

\fi

\bibliographystyle{alpha}
\ifabstr
{\footnotesize\bibliography{ief}}
\else 
\bibliography{ief}
\fi
\ifabstr

%\section*{Appendix}
%\medskip
%What follows is the full version of this paper that include all proofs, as well as
%some additional remarks and references. 
%of Lemma~\ref{lem:system}, Proposition~\ref{prop:sufficient} and
%Lemma~\ref{c:ineq}, as well as an additional remark after the proof of
%Lemma~\ref{lem:system}.  
\fi

\end{document}